\documentclass{amsart}
\usepackage[margin=1.5cm]{geometry}
\numberwithin{equation}{section}
\usepackage{amssymb}
\usepackage{amsmath, amsfonts,amsthm,amssymb,amscd, verbatim,graphicx,color,multirow,booktabs, caption,tikz,tikz-cd, mathdots,bm}
\usepackage{tikz-cd}
 \usepackage{hyperref}
\usetikzlibrary{positioning}

\newtheorem{theorem}{Theorem}[section]
\newtheorem{lemma}[theorem]{Lemma}

\newtheorem{proposition}[theorem]{Proposition}

      \theoremstyle{definition}

     \theoremstyle{remark}
     \newtheorem{remark}[theorem]{Remark}

\newcommand{\Sym}{\mathop{\mathrm{Sym}}}
\newcommand{\Alt}{\mathop{\mathrm{Alt}}}

 \definecolor{mycolor}{rgb}{0.55,0.0,0.16}
  \definecolor{myred}{rgb}{0.6,0.0,0.16}
  \definecolor{mygreen}{rgb}{0.0,0.6,0.16}
  \definecolor{myviolet}{rgb}{1,0,1}

\hypersetup{
colorlinks=false,
allbordercolors=myred,
citebordercolor=mygreen
}

\makeatletter
\@namedef{subjclassname@2020}{%
  \textup{2020} Mathematics Subject Classification}
\makeatother

\begin{document}
\title[Abelian sections of the symmetric groups]{Abelian sections of the symmetric groups\\with respect to their index}
\author[L. Sabatini]{Luca Sabatini}
\address{Luca Sabatini, Dipartimento di Matematica  e Informatica ``Ulisse Dini'',\newline
 University of Firenze, Viale Morgagni 67/a, 50134 Firenze, Italy} 
\email{luca.sabatini@unifi.it}
	\maketitle

        \begin{abstract}
We show the existence of an absolute constant $\alpha>0$ such that, for every $k \geq 3$, $G:= \Sym(k)$,
        and for every $H \leqslant G$ of index at least $3$,
       one has $|H/[H,H]| \leq |G:H|^{\alpha/ \log \log |G:H|}$.
        This inequality is the best possible for the symmetric groups,
        and we conjecture that it is the best possible for every family of arbitrarily large finite groups.
\end{abstract}

\vspace{0.1cm}
\section{Introduction}

Abelian quotients of permutation groups attracted attention for the first time in \cite{1989KP},
where the authors show that an abelian section of $\Sym(k)$ has order at most $3^{k/3}$, for every $k \geq 3$.
Better bounds hold for primitive groups \cite{AG}, and for transitive groups \cite{2020LSS}.
In these notes, a different aspect concerning the subgroups of the symmetric groups is revealed:
as the index increases, the abelian quotients grow as slowly as possible.
 
     \begin{theorem} \label{thMain}
     There exists an absolute constant $\alpha >0$ such that for every $k \geq 1$
     and every $H \leqslant G := \Sym(k)$ of index at least $3$
     one has
      $$ |H/H'| \> \leq \> |G:H|^{\> \alpha / \log \log |G:H|} , $$
      where $H':=[H,H]$ denotes the commutator subgroup of $H$.
     \end{theorem}
     \vspace{0.1cm}
    
     This bound is sharp in a number of situations, which make the proof by induction somewhat challenging.
     For example, equality is satisfied for infinitely many $k$,
     by elementary abelian groups of order $p^{\lfloor k/p \rfloor}$ having all of their orbits of cardinality $p$.
     More is said in Section \ref{sectFinal}, where we show some evidences towards the fact that Theorem \ref{thMain}
     is the best possible for every family of arbitrarily large finite groups.

   \vspace{0.1cm}
   \section{Preliminaries} \label{Sect2}
   
Unless explicitly stated otherwise, all the logarithms are to base $2$, and $\exp(x):=2^x$.
We will often use inequalities for the factorial function: to avoid useless calculations,
we recall three estimates, which provide increasing accuracy.
   
    \begin{lemma}[Factorial estimates] \label{3lemFactorial}
     Let $k \geq 1$. Then
     
     \begin{itemize}
     \item[(i)]
     $$ k^{k/2} \> \leq \> k! \> \leq \> k^k . $$
     \item[(ii)] 
      $$ \frac{k^k}{e^{k-1}} \leq \> k! \> \leq \frac{k^{k+1}}{e^{k-1}} . $$
      \item[(iii)]
      $$ \left( \frac{k}{e} \right)^k \sqrt{2 \pi k} \> \leq \> k! \> \leq \> \left( \frac{k}{e} \right)^k \sqrt{2 \pi k} \cdot e^{1/12k} . $$
     \end{itemize}
     
     \end{lemma}
     \begin{proof}
     The right side of (i) is obvious and the left side is equivalent to $\sqrt{k} \leq (k!)^{1/k}$.
     These are the geometric means of $\{1,k\}$ and $\{1,2,...,k\}$ respectively.
     Since the product $j(k-j)$ is maximum where $j$ is close to $k/2$, then the first mean is at most the second, as desired.\\
    For (ii), let us notice that
    $$ \frac{k^k}{k!} = \prod_{j=1}^{k-1} \frac{(j+1)^j}{j^j} ,
    \hspace{1cm} \mbox{and} \hspace{1cm} 
    \frac{k!}{k^{k+1}} = \prod_{j=1}^{k-1} \frac{j^{j+1}}{(j+1)^{j+1}} . $$
    To obtain the left side, we use $\tfrac{j+1}{j} \leq e^{1/j}$, so that
    $$ \prod_{j=1}^{k-1} \frac{(j+1)^j}{j^j} \leq \prod_{j=1}^{k-1} e = e^{k-1} . $$
    To obtain the right side, we use $\tfrac{j}{j+1} \leq e^{-1/(j+1)}$, so that
    $$ \prod_{j=1}^{k-1} \frac{j^{j+1}}{(j+1)^{j+1}} \leq \prod_{j=1}^{k-1} e^{-1} = \frac{1}{e^{k-1}} . $$
    Finally, (iii) is the sharp estimate of Robbins \cite{1955Rob}.
     \end{proof}
    \vspace{0.1cm}
     
     As it is easy to see, taking logarithms in Lemma \ref{3lemFactorial} (i), we see that
     $$ \tfrac{1}{2} k \log k \> \leq \> \log(k!) \> \leq \> k \log k , $$
     and that
     $$ \tfrac{1}{2} \log k \> \leq \> \log \log (k!) \> \leq \> 2 \log k $$
     for all $k \geq 3$.

  \begin{lemma} \label{2lemDSDP}
    Let $G$ be an arbitrary group, $N \lhd G$, and $G/N \cong Q$.
    Then 
    $$ |Q/Q'| \> \leq \> |G/G'| \> \leq \> |Q/Q'| \cdot |N/N'| . $$
    \end{lemma}    
    \begin{proof}
    We have
    $$ |G/G'| = |G:NG'| \cdot |NG': G'| = |(G/N):(G/N)'| \cdot |N : N \cap G'| . $$
    Since $N' \subseteq N \cap G'$, the proof follows.
    \end{proof}
    \vspace{0.1cm}
  
   To prove Theorem \ref{thMain}, we seek for an absolute constant $C>0$ such that,
	for every sufficiently large $k$ and every $H \leqslant \Sym(k)$ of index at least $3$, one has
	\begin{equation} \label{3eqAbGSym}
	\log|H/H'| \> \leq \> C \cdot \frac{\log|\Sym(k):H|}{\log \log|\Sym(k):H|} .
	\end{equation}
	
  By $f(n) \ll g(n)$ and $g(n) \gg f(n)$
we mean the same thing, namely, that there is $C>0$ such that $f(n) \leq C \cdot g(n)$ for all $n \geq C$.
 Moreover, we will use frequently the fact that the function $x/\log x$ is an increasing function when $x \geq 3$.

 \begin{remark}[Small subgroups] \label{remSmall}
	 Choose a large $k$, and let $G:=\Sym(k)$.
	We notice that (\ref{3eqAbGSym}) is true for subgroups $H$ such that $\log|H| \leq \tfrac{1}{4} k \log k$.
In fact, via Lemma \ref{3lemFactorial} (i), we have
$$\log|G:H| \geq \log(k!) - \tfrac{1}{4} k \log k \geq \tfrac{1}{4} k \log k . $$
From the main theorem of \cite{1989KP} we have $|H/H'| \leq 3^{k/3}$, and so
$$ \log|H/H'| \leq 2k \leq
16 \cdot \frac{\tfrac{1}{4} k \log k}{\log(\tfrac{1}{4} k \log k)} \leq 16 \cdot \frac{\log|G:H|}{\log \log|G:H|} . $$
From the main theorem of \cite{1980PS}, a primitive subgroup of $G$ which does not contain $\Alt(k)$ has size at most $4^k$.
Since $\Alt(k)$ is perfect, primitive groups do not affect the proof of Theorem \ref{thMain}.
\end{remark}

\vspace{0.1cm}
   \section{Transitive subgroups}
   
     In short, looking inside $\Sym(k)$, we use two sharp inequalities with respect to $k$:
     one for the index of a maximal transitive group, and one for the abelianization of a transitive group,
     which is provided by \cite{2020LSS}.
     The structure of the groups we are left with is trapped by some smaller symmetric or alternating group,
     and this allows to argue that their abelian quotients are small.
     
   \begin{lemma} \label{3lemTransitiveSize}
     Let $W := \Sym(k/b)^b \rtimes \Sym(b) \leqslant G$, for some $2 \leq b \leq k/2$ which divides $k$.
     Then
     $$ \log|G:W| \> \geq \> \frac{k \log b}{3} $$
     if $k$ is sufficiently large.
     \end{lemma}
     \begin{proof}
     We will prove the same inequality with the natural logarithm in place of $\log$.
     Taking the logarithm in Lemma \ref{3lemFactorial} (ii), we obtain
     $$ k \ln k - (k-1) \leq \ln(k!) \leq (k+1) \ln k -(k-1) . $$
     Using three times these inequalities, we can write
     $$ \ln|G:W| = \ln(k!) - \ln|W| = \ln(k!) -\ln(b!) -b \ln((k/b)!) \geq $$
  $$ k \ln k -(k-1) -(b+1) \ln b +(b-1) -b(k/b+1) \ln(k/b) +b(k/b-1) = $$
  $$ k \ln k -k -b \ln b -\ln b -k \ln(k/b) -b \ln(k/b) +k = $$
  $$ k \ln k -k -b \ln b -\ln b -k \ln k + k \ln b -b \ln k +b \ln b +k = $$
  $$ (k-1) \ln b -b \ln k . $$
  It rests to prove that this is at least $\tfrac{1}{3} k \ln b$, and arranging the terms we see that this is equivalent to
  $$ \frac{\tfrac{2}{3} k -1}{\ln k} \geq \frac{b}{\ln b} . $$
  Now we observe that, for large enough $k$ we have
  \begin{equation} \label{3eqTrans}
  \frac{\tfrac{2}{3} k -1}{\ln k} \geq \frac{(k/2)}{\ln (k/2)} ,
  \end{equation}
  because $\left( \tfrac{2}{3}k-1 \right)/(k/2)$ converges to $4/3$, while $\ln(k)/\ln(k/2)$ converges to $1$.
  Finally, for every $2 \leq b \leq k/2$, the right side of (\ref{3eqTrans}) is at least $b/\ln b$,
  because $x/\ln x$ is an increasing function.
     \end{proof}

   \begin{theorem}[Theorem 1 in \cite{2020LSS}] \label{3thLSS}
     If $G \leqslant \Sym(k)$ is a {\bfseries transitive} permutation group, then
     $$ |G/G'| \> \leq \> 4^{k/\sqrt{\log k}} . $$
     \end{theorem}
     \vspace{0.1cm}
     
     In reality, the proof of Theorem \ref{thMain} requires a slightly better result than Theorem \ref{3thLSS} itself.
     Given a finite group $R$ and a prime $p \geq 2$, let $a_p(R)$ be the number of the abelian composition factors of $R$ of order $p$.
     We define
     $$ a(R) \> := \> \sum_{p\,\mathrm{prime}} a_p(R) \cdot \log p . $$
	Informally, this is the logarithm of the ``abelian portion'' of $R$.
	We also introduce some more notation about wreath products.
	Let $W:=R^b \rtimes \Sym(b)$.
	We denote by $\rho_{\Sym(b)} : W \rightarrow \Sym(b)$ the projection over the top group,
	and for every $j =1,...,b$, we denote by $\prod_{j=1}^b R_j$ the base subgroup. For every $j$ we have
	$$ N_W(R_j) =R_j \times ( R^{b-1} \rtimes \Sym(b-1)) . $$
	This allows to consider the projections $\rho_j : N_W(R_j) \rightarrow R_j$.

     \begin{proposition} \label{3propLSSGen}
     Let $G \leqslant R \wr_k \Sym(k)$ such that $\rho(G)$ is transitive and
     $\rho_j(N_G(R_j))=R_j$ for every $j=1,..,k$.
     Then
     $$ \log|G/G'| \> \leq \> (1+a(R)) \frac{2k}{\sqrt{\log k}} . $$
     \end{proposition}
     \begin{proof}
     As an intermediate result towards Theorem \ref{3thLSS}, \cite[Lemma 3.1]{2020LSS} says that the inequality
     $$ \log |G/G'| \leq \frac{(2/\sqrt{\pi}) \cdot a(R) k}{\sqrt{\log k}} + \log |\rho(G)/(\rho(G))'| $$
     is true under our hypothesis. Since $\rho(G) \leqslant \Sym(k)$ is transitive, putting back Theorem \ref{3thLSS},
     and noting that $2/\sqrt{\pi} <2$, we obtain the claimed inequality.
     \end{proof}

     \begin{proof}[Proof of Theorem \ref{thMain} for transitive $H$]
We will always suppose that $k$ is larger than any constant.
	We have already settled primitive groups at the end of Section \ref{Sect2},
     so let $H \leqslant G$ be transitive but not primitive,
     and contained in a maximal transitive group $W:= \Sym(a)^b \rtimes \Sym(b)$ as in Lemma \ref{3lemTransitiveSize}.
     In particular, we choose $W$ in such a way that $a$ is the smallest possible
     (equivalently, the blocks of imprimitivity have minimal size).
     Then $\log|G:H| \geq \log|G:W| \geq \tfrac{1}{3} k \log b$.
    If $\log b \geq \sqrt{\log k}$, from Theorem \ref{3thLSS} we have
    $$ \log|H/H'| \leq \frac{2k}{\sqrt{\log k}} \ll \frac{\tfrac{1}{3} k \sqrt{\log k}}{\log(\tfrac{1}{3} k \sqrt{\log k})} \leq
    \frac{\tfrac{1}{3} k \log b}{\log(\tfrac{1}{3} k \log b)} \leq \frac{\log|G:H|}{\log\log|G:H|} , $$
 and then (\ref{3eqAbGSym}) is true.
Thus, it rests to control {\bfseries all} the cases
\begin{equation} \label{3eqSmallB}
1 \leq \log b \leq \sqrt{\log k} . 
\end{equation}
For such a fixed $b$, we take a closer look at $H \leqslant W$.
Let us recall the notation we introduced just before the statement of Proposition \ref{3propLSSGen}.
First, the projection over the top group $\rho_{\Sym(b)}(H) \leqslant \Sym(b)$ is transitive
(otherwise $H \leqslant \Sym(k)$ itself would be not transitive).
This implies that the projections $\rho_j$ of $N_H(\Sym(a))$ in $\Sym(a)$ are all isomorphic for every $1 \leq j \leq b$.
Let us denote by $H_{proj} \leqslant \Sym(a)$ one of these projections.
  Since the blocks of imprimitivity have minimal size by the construction of $W$,
  we get that $H_{proj} \leqslant \Sym(a)$ is primitive.
If $H_{proj}$ does not contain $\Alt(a)$,
then from the main theorem of \cite{1980PS} we have $|H_{proj}| \leq 4^a$.
By the imprimitive embedding theorem we have $H \leqslant (H_{proj})^b \rtimes \Sym(b)$, and so $|H| \leq 4^{ab} \cdot b!$.
Using also Lemma \ref{3lemFactorial} (i), and (\ref{3eqSmallB}), it follows that
$$ \log|G:H| \geq \tfrac{1}{2} k \log k -2k -b \log b \gg  k \log k - (\sqrt{\log k}) 2^{\sqrt{\log k}} \gg k \log k . $$
As we have seen in Remark \ref{remSmall}, this implies that $H$ is too much small.\\
We are left with the cases where either $H_{proj}=\Alt(a)$ or $H_{proj}=\Sym(a)$.
If $H_{proj} = \Alt(a)$, then from Proposition \ref{3propLSSGen} (notice that $a(R)=0$ in this case) we have
$$ \log|H/H'| \leq \frac{2b}{\sqrt{\log b}} . $$
Using again (\ref{3eqSmallB}), we obtain
$$ \log|H/H'| \leq \frac{2b}{\sqrt{\log b}} \ll \frac{2^{\sqrt{\log k}}}{(\log k)^{1/4}} \ll 
\frac{k}{\log k} \ll \frac{\log|G:W|}{\log\log|G:W|} \leq \frac{\log|G:H|}{\log\log|G:H|} . $$
If $H_{proj} = \Sym(a)$, then from Proposition \ref{3propLSSGen} (notice that $a(R)=1$ in this case) we have
$$ \log|H/H'| \leq \frac{4b}{\sqrt{\log b}} . $$
Using a last time (\ref{3eqSmallB}), we obtain
$ \log|H/H'| \ll \frac{\log|G:H|}{\log\log|G:H|} $ 
as before, and the proof of the transitive case is complete.
\end{proof}

   \vspace{0.1cm}
   \section{Intransitive subgroups}
   
   Let $H$ be contained in $\Sym(a) \times \Sym(b)$ for some $a+b=k$ and $1 \leq b \leq a \leq k-1$. 
   Consider the projection $\rho: H \rightarrow \Sym(b)$.
   The factorized subgroup $Ker(\rho) \times \rho(H) \leqslant \Sym(a) \times \Sym(b)$
   has the same size of $H$, and not smaller abelianization from Lemma \ref{2lemDSDP}.
   Thus, we can suppose $H=A \times B$ for some $A \leqslant \Sym(a)$ and $B \leqslant \Sym(b)$.
    We can also suppose that $a$, as $k$, is larger than any constant.
      We have
      $$ \log|H/H'| = \log|A/A'| + \log|B/B'| , $$
      and now we go by induction distinguishing two cases with respect to $b$.
      
      \subsection{Small $b$}
      First we suppose $b \leq M$, where $M>0$ is a large positive integer to be fixed later.
      If $a!/|A| < a$, then $|H/H'| \leq M!$, and (\ref{3eqAbGSym}) follows easily (with some $C$ depending on $M$).
      If $a!/|A| \geq a$, then
      by induction we have $\log|A/A'| \leq C \tfrac{\log(a!/|A|)}{\log \log(a!/|A|)}$.
      So, it is enough to prove that for some $C(M)$ and every $a!/|A| \geq a$ one has 
      \begin{equation} \label{eqC}
      C \frac{\log(a!/|A|)}{\log \log (a!/|A|)} + M! \> \leq \>
      C \frac{\log \left( \tfrac{a!}{|A|} \cdot \tfrac{(a+1) \cdot ... \cdot (a+b)}{b!} \right)}
      {\log \log \left( \tfrac{a!}{|A|} \cdot \tfrac{(a+1) \cdot ... \cdot (a+b)}{b!} \right)} 
	\end{equation}    
      for all $a$ large enough.
      Set $x:=a!/|A|$.
      We notice that, for all $a \geq 3$,
      $$ \tfrac{(a+1)\cdot ... \cdot (a+b)}{b!} = {a+b \choose a} \geq a \geq \frac{\log (a!)}{2 \log\log (a!)} \geq \frac{\log x}{2 \log \log x} . $$
      We can assume $x \geq a$. Now we compute the following limit.
      
      \begin{lemma} \label{lemmaLimit} 
      $$ \lim_{x \rightarrow +\infty} \> \left(
      \frac{\log \left( \frac{x \log x}{2 \log \log x} \right)}{\log \log \left( \frac{x \log x}{2 \log \log x} \right)}
      - \frac{\log x}{\log \log x} \right) \> = \> 1 .  $$
     \end{lemma}
      \begin{proof}
    We have
    $$ 
      \frac{\log \left( \frac{x \log x}{2 \log \log x} \right)}{\log \log \left( \frac{x \log x}{2 \log \log x} \right)}
      - \frac{\log x}{\log\log x} = $$
   $$  \left( \frac{\log x}{\log \log \left( \frac{x \log x}{2 \log \log x} \right)} - \frac{\log x}{\log\log x} \right) + 
      \frac{\log\log x}{\log \log \left( \frac{x \log x}{2 \log \log x} \right)} -
      \frac{\log(2\log\log x)}{\log \log \left( \frac{x \log x}{2 \log \log x} \right)} . $$
      Then, for $x \rightarrow +\infty$, it is easy to see that the second term converges to $1$, while the third term converges to zero.
      For the first term, we have that this is equal to
      $$ \frac{\log x}{\log \log \left( \frac{x \log x}{2 \log \log x} \right)} 
      \left( 1- \frac{\log \log \left( \frac{x \log x}{2 \log \log x} \right)}{\log \log x} \right) 
      \rightarrow 0 . \qedhere $$
    \end{proof}
     \vspace{0.1cm}
      
      From Lemma \ref{lemmaLimit}, we have that 
      $$ \frac{\log \left( \tfrac{a!}{|A|} \cdot \tfrac{(a+1) \cdot ... \cdot (a+b)}{b!} \right)}
      {\log \log \left( \tfrac{a!}{|A|} \cdot \tfrac{(a+1) \cdot ... \cdot (a+b)}{b!} \right)}  
      - \frac{\log(a!/|A|)}{\log \log (a!/|A|)} \geq 1/2 $$
     is true for every $a$ large enough.
      Thus, for $C:=2 M!$, we obtain that (\ref{eqC}) is true for every $a$ large enough, as desired.\\

     \subsection{Large $b$}
      Let us suppose $a \geq b > M$, where $M>0$ is again a large positive integer to be fixed later.
      By induction, we have
      $$ \log|A/A'| + \log|B/B'| \> \leq \>
      \frac{\log(a!/|A|)}{\log \log (a!/|A|)} + \frac{\log(b!/|B|)}{\log \log (b!/|B|)} . $$
 	Let $x:=a!/|A|$ and $y:=b!/|B|$.
 	Since $|G:H|=\tfrac{(a+b)!}{|A||B|}= xy{a+b \choose a}$, we need to prove that
 	\begin{equation} \label{eqDOIT}
 	\frac{\log x}{\log \log x} + \frac{\log y}{\log \log y}
      \> \leq \>
      \frac{\log (xy {a+b \choose a})}{\log \log (xy {a+b \choose a})}
 	\end{equation}
 	is true for every $a \geq b >M$, $a \leq x \leq a!$, $b \leq y \leq b!$.\\
    
    \begin{lemma}
    Let $X,Y,K$ be positive integers larger than $2$.
    If 
    $$ \frac{\log X}{\log \log X} + \frac{\log Y}{\log \log Y}
      \> \leq \>
     \frac{\log (XY \cdot K)}{\log\log(XY \cdot K)} , $$
     then 
       $$ \frac{\log x}{\log \log x} + \frac{\log y}{\log \log y}
      \> \leq \>
     \frac{\log (xy \cdot K)}{\log\log(xy \cdot K)} $$
     for every $3 \leq x \leq X$ and $3 \leq y \leq Y$.
    \end{lemma}
    \begin{proof}
    We will argue replacing $x,y,X,Y,K$ with their logarithms (to the base $2$). 
    Fix $K \geq \log 3$, and set
    $$ f(x,y) \> := \> 
    \frac{x+y+K}{\log(x+y+K)} -  \frac{x}{\log x} - \frac{y}{\log y} . $$
    We will prove that $f(x,y)$ is non-increasing in $x$ and $y$.
    To do this, we can replace $\log_2$ with $\ln$ in the definition of $f$.
    When considered in $(1,+\infty) \times (1,+\infty)$, $f(x,y)$ is an analytic function.
    Computing the partial derivative with respect to $x$, we obtain
    $$ \frac{\partial f}{\partial x} 
    \leq - \frac{1}{\ln x} +  \frac{1}{(\ln x)^2} + \frac{1}{\ln(x+y+K)} \leq 0 . $$
    Since the expression of $f$ is symmetric with respect to $x$ and $y$,
    we have $\frac{\partial f}{\partial y} \leq 0$ as before,
    and the proof follows.
    \end{proof}
 \vspace{0.1cm}
   
   From the previous lemma, it is enough to check (\ref{eqDOIT}) when $x=a!$ and $y=b!$.
   The next inequality is really about the inverse function of the gamma function,
   and concludes the proof of Theorem \ref{thMain}.
   
   	\begin{proposition}
   	There exists an absolute constant $M >0$ such that, whenever $M \leq b \leq a$, then
   	$$ \frac{\log (a!)}{\log \log (a!)} + \frac{\log (b!)}{\log \log (b!)}
      \> \leq \>
      \frac{\log ((a+b)!)}{\log \log ((a+b)!)} . $$
   	\end{proposition}
 	\begin{proof}
 	Taking the natural logarithm in Lemma \ref{3lemFactorial} (iii), we obtain
 	  $$ k(\ln k -1) + \frac{\ln(2 \pi k)}{2} \leq \ln(k!) \leq k(\ln k -1) + \frac{\ln(2 \pi k)}{2} + \frac{1}{12k} . $$
 	  Then
 	$$ \frac{\log (a!)}{\log \log (a!)} + \frac{\log (b!)}{\log \log (b!)} = 
 	\frac{\ln (a!)}{\ln \left( \frac{\ln (a!)}{\ln 2} \right) } + \frac{\ln (b!)}{\ln \left( \frac{\ln (b!)}{\ln 2} \right) } \leq $$
 	$$ \frac{a(\ln a -1) + \frac{\ln(2 \pi a)}{2} + \frac{1}{12a}}{\ln \left( \tfrac{1}{\ln 2} (a(\ln a -1) + \frac{\ln(2 \pi a)}{2} + \frac{1}{12a}) \right)}
 	  + \frac{b(\ln b -1) + \frac{\ln(2 \pi b)}{2} + \frac{1}{12b}}{\ln \left( \tfrac{1}{\ln 2} ( b(\ln b -1) + \frac{\ln(2 \pi b)}{2} + \frac{1}{12b}) \right)} 
 	  \leq $$
 	  $$ \frac{(a+b)(\ln (a+b) -1) + \frac{\ln(2 \pi (a+b))}{2}}{\ln \left( \tfrac{1}{\ln 2} ((a+b)(\ln (a+b) -1) + \frac{\ln(2 \pi (a+b))}{2} ) \right)} \leq $$
 	$$ \frac{\ln ((a+b)!)}{\ln \left( \frac{\ln ((a+b)!)}{\ln 2} \right) } = \frac{\log ((a+b)!)}{\log \log ((a+b)!)} . $$
 	Indeed, the inequality in the middle is true for sufficiently large $a$ and $b$,
 	because comparing the leading terms in the asymptotic expansions of both sides we obtain
 	$$ a \left( 1-\frac{1}{\ln a} \right) + b \left( 1-\frac{1}{\ln b} \right) \leq 
 	(a+b) \left( 1-\frac{1}{\ln(a+b)} \right) . \qedhere $$
 	\end{proof}

    \vspace{0.1cm}
    \section{Arbitrary finite groups} \label{sectFinal}
    
   	From Lemma \ref{3lemFactorial} (i) we have
     $$ k = \frac{k \log k}{\log k} \geq \frac{\log(k!)}{2 \log \log(k!)} . $$
    Thus, for every $k$ which is a multiple of $3$,  $G:= \Sym(k)$, and $H \leqslant G$ an elementary abelian $3$-group of size $3^{k/3}$,
     we obtain
     \begin{equation} \label{eqSharp}
      |H/H'| = 2^{\tfrac{k \log 3}{3}} \geq 
     \exp \left( \frac{(\log3/6) \log|G|}{\log \log|G|} \right) \geq \exp \left( \frac{(\log3/6) \log |G:H|}{\log \log |G:H|} \right) . 
     \end{equation}
     This shows that Theorem \ref{thMain} is the best possible for the symmetric groups.
     When $G$ is an arbitrary finite group, we have the following.
    
    \begin{proposition} \label{propAbSect}
    Every finite group $G$ of size at least $3$ has an abelian section of size at least $|G|^{1/6(\log \log |G|)^2}$.
    \end{proposition}
    \begin{proof}
    It is well known that every group of size $p^n$ has derived length at most $\log n$.
    Using pigeonhole on the derived series, we see that such a group has an abelian section of size at least $p^{n/\log n}$.
    Thus, every nilpotent group $H$ of size $p_1^{n_1} \cdot ... \cdot p_k^{n_k}$ has an abelian section of size at least
    $p_1^{n_1/\log n_1} \cdot ... \cdot p_k^{n_k/\log n_k}$. Since $\log n_i \leq \log \log |H|$ for every $i=1,...,k$,
    it follows that this size is at least $|H|^{1/ \log \log |H|}$.
    Now, a result of Pyber \cite[Corollary 2.3 (a)]{1997Pyber}
    shows that every finite group $G$ has a solvable subgroup of size at least $|G|^{1/2 \log \log |G|}$.
    By another result of Heineken \cite[Corollary]{1991Heineken}, a finite solvable group $S$
    has a nilpotent subgroup of size at least $|S|^{1/3}$.
    Then, putting all together, an arbirary finite group $G$ has an abelian section of size at least
    $$ |H|^{1/\log\log|H|} \geq |S|^{1/3 \log \log |S|} \geq |G|^{1/6(\log \log|G|)(\log \log |S|)} \geq
    |G|^{1/3(\log \log |G|)^2} , $$
    where we used that $x^{1/2 \log \log x}$ is an increasing function when $x \geq 7$.
    An analysis of groups of small order concludes the proof.
    \end{proof}
    \vspace{0.1cm}
    
    Arguing as in (\ref{eqSharp}), Proposition \ref{propAbSect} shows that
    Theorem \ref{thMain} is not far from the best possible, for every family of arbitrarily large finite groups.
     It is an intriguing question whether the $2$ at the exponent in Proposition \ref{propAbSect} can be removed:
     in fact, we conjecture that this can be done.
    Finally, it is worth to notice that a positive answer to a question of Pyber \cite[Problem 14.76]{2018Kourovka}
     would imply such an improvement of Proposition \ref{propAbSect},
     showing again that Theorem \ref{thMain} is the best possible for every family of arbitrarily large finite groups.\\
     
     {\bfseries P.S.} The article \cite{2022Sab} provides the conjectured improvement to Proposition \ref{propAbSect}.


\vspace{0.3cm}
\begin{thebibliography}{1}

\bibitem{AG} M.~Aschbacher, R.~M.~Guralnick, \textit{On abelian quotients of primitive groups},
Proceedings of the American Mathematical Society \textbf{107} (1989), 89-95.

\bibitem{1991Heineken} H. Heineken, \textit{Nilpotent subgroups of finite solvable groups},
Archiv der Mathematik \textbf{56} (1991), 417-423.

\bibitem{2018Kourovka} E.I. Khukhro, V.D. Mazurov, \textit{Unsolved Problems in Group Theory: The Kourovka Notebook}
\textbf{(19)}, Sobolev Institute of Mathematics (2018).

 \bibitem{1989KP} L. Kov\'acs, C. Praeger, \textit{Finite permutation groups with large abelian quotients},
Pacific Journal of Mathematics \textbf{136 (2)} (1989), 283-292.

	\bibitem{2020LSS} A. Lucchini, L. Sabatini, P. Spiga, \textit{A subexponential bound on the cardinality
of abelian quotients in finite transitive groups},
 to appear in Bulletin of the London Mathematical Society (2021).
	
	\bibitem{1980PS} C. Praeger, J. Saxl, \textit{On the orders of primitive permutation groups},  
 	Bulletin of the London Mathematical Society \textbf{12 (4)} (1980), 303-307.
 	
 	\bibitem{1997Pyber} L. Pyber, \textit{How abelian is a finite group?},
The Mathematics of Paul Erd\H{o}s (1997), 372-384.
   
   \bibitem{1955Rob} H. Robbins, \textit{A remark on Stirling's formula},  
   The American Mathematical Monthly \textbf{62 (1)} (1955), 26-29.
   
   \bibitem{2022Sab} L. Sabatini, \textit{Nilpotent subgroups of class $2$ in finite groups},
 to appear in Proceedings of the American Mathematical Society (2022).
	
	

\vspace{1cm}

\end{thebibliography}
\end{document}